\documentclass{amsart}

\usepackage{amssymb}
\usepackage{microtype}
\usepackage{fullpage}
\usepackage{hyperref}

\hypersetup{colorlinks=true, citecolor=blue, linkcolor=blue, urlcolor=blue, pdfstartview=FitH, pdfauthor=Valentin Blomer and Gergely
Harcos and P\'eter Maga, pdftitle=Analytic properties of spherical cusp forms on GL(n)}

\newcommand{\ZZ}{\mathbb{Z}}
\newcommand{\QQ}{\mathbb{Q}}
\newcommand{\RR}{\mathbb{R}}
\newcommand{\CC}{\mathbb{C}}
\newcommand{\aA}{\mathbb{A}}
\newcommand{\cH}{\mathcal{H}}
\newcommand{\cW}{\mathcal{W}}
\newcommand{\cB}{\mathcal{B}}
\newcommand{\cD}{\mathcal{D}}
\newcommand{\cY}{\mathcal{Y}}

\newcommand{\eps}{\varepsilon}
\newcommand{\bs}{\backslash}
\renewcommand{\geq}{\geqslant}
\renewcommand{\leq}{\leqslant}

\newcommand{\GL}{\mathrm{GL}}
\newcommand{\PGL}{\mathrm{PGL}}
\newcommand{\SL}{\mathrm{SL}}
\renewcommand{\O}{\mathrm{O}}

\newcommand{\Z}{\mathrm{Z}}
\newcommand{\U}{\mathrm{U}}
\newcommand{\M}{\mathrm{M}}
\newcommand{\ov}[1]{\overline{#1}}

\DeclareMathOperator{\Ad}{Ad}
\DeclareMathOperator{\diag}{diag}
\DeclareMathOperator{\vol}{vol}
\DeclareMathOperator{\height}{ht}

\theoremstyle{plain}
\newtheorem{lemma}{Lemma}
\newtheorem{theorem}{Theorem}

\theoremstyle{remark}
\newtheorem{remark}{Remark}

\theoremstyle{definition}
\newtheorem*{acknowledgement}{Acknowledgements}

\begin{document}

\author{Valentin Blomer}
\author{Gergely Harcos}
\author{P\'eter Maga}

\address{Mathematisches Institut, Endenicher Allee 60, 53115 Bonn, Germany}\email{blomer@math.uni-bonn.de}
\address{Alfr\'ed R\'enyi Institute of Mathematics, Hungarian Academy of Sciences, POB 127, Budapest H-1364, Hungary}\email{gharcos@renyi.hu, magapeter@gmail.com}
\address{MTA R\'enyi Int\'ezet Lend\"ulet Automorphic Research Group}\email{gharcos@renyi.hu, magapeter@gmail.com}
\address{Central European University, Nador u. 9, Budapest H-1051, Hungary}\email{harcosg@ceu.edu}

\title{Analytic properties of spherical cusp forms on $\GL(n)$}

\dedicatory{Dedicated to Dorian Goldfeld on the occasion of his seventy-first birthday}

\begin{abstract}
Let $\phi$ be an $L^2$-normalized spherical vector in an everywhere unramified cuspidal automorphic representation of $\PGL_n$ over $\QQ$ with  Laplace eigenvalue $\lambda_{\phi}$. We establish explicit estimates for various quantities related to $\phi$ that are uniform in $\lambda_{\phi}$. This includes uniform bounds for spherical Whittaker functions on $\GL_n(\RR)$, uniform bounds for the global sup-norm of $\phi$, and uniform bounds for the ``essential support'' of $\phi$, i.e.\ the region outside which it decays exponentially. The proofs combine analytic and arithmetic tools.
\end{abstract}



\thanks{First author partially supported by the DFG-SNF lead agency program grant BL~915/2-2. Second and third author supported by NKFIH (National Research, Development and Innovation Office) grants NK~104183, ERC\underline{\phantom{ }}HU\underline{\phantom{ }}15~118946, K~119528, and by the MTA R\'enyi Int\'ezet Lend\"ulet Automorphic Research Group. Second author also supported by ERC grant AdG-321104, and third author also supported by the Premium Postdoctoral Fellowship of the Hungarian Academy of Sciences.}

\maketitle

\section{Introduction}

Classical modular forms for congruence subgroups of $\SL_2(\ZZ)$ have a long tradition in many branches of mathematics, in particular number theory. The familiar framework of $2$-by-$2$ matrices and the corresponding symmetric space of the Poincar\'e upper plane is amenable to concrete computations and explicit formulae. For instance, depending on the choice of coordinates, the eigenfunctions of the Laplace operator can be expressed in terms of Bessel functions or hypergeometric functions that have been studied extensively and are, by and large, well-understood.

The analytic picture changes completely for automorphic forms on higher rank groups, where the complexity increases so drastically that explicit results suitable for the purpose of analytic number theory often remain elusive. In this work, we focus on cusp forms for the group $\PGL_n$ over $\QQ$ that come with local Langlands parameters at each place of $\QQ$. We keep the cusp form unramified (spherical) at all places, but single out the archimedean place and investigate the analytic properties as the maximal archimedean Langlands parameter (or equivalently the Laplace eigenvalue) grows. Thus our key players are the real-analytic functions $\phi$ on the non-compact locally symmetric space
\[X_n:=\GL_n(\ZZ)\Z_n(\RR)\bs\GL_n(\RR)/\O_n(\RR)\]
that are of moderate growth, eigenfunctions of the commuting family $\cD_n$ of all invariant differential operators on $X_n$, and satisfy the cuspidality condition
\begin{equation}\label{cusp}
\int_{\U(\ZZ)\bs\U(\RR)} \phi(u g) \, du = 0,\qquad g\in\GL_n(\RR),
\end{equation}
for all unipotent block upper triangular subgroups $\U$ of $\GL_n$ (cf.\ \cite[Def.~5.1.3]{G}). Here $\Z_n(\RR)$ denotes the center of $\GL_n(\RR)$, and $\O_n(\RR)$ denotes the orthogonal subgroup, while $\U_n(\RR)$ will be reserved for the subgroup of unipotent upper triangular matrices. A particular element of $\cD_n$ is the Laplace operator on $X_n$, and we denote the corresponding eigenvalue of $\phi$ by $\lambda_\phi$.

These eigenfunctions are the  building blocks of the cuspidal spectrum  of $L^2(X_n)$, and therefore they are of central importance in analysis. Being spherical vectors of  cuspidal  automorphic representations, cusp forms $\phi$ are in addition eigenfunctions of the global Hecke algebra, but for most of time we do not assume this extra property. (It is not unreasonable to conjecture that the eigenspaces of $\cD_n$ are one-dimensional so that the Hecke property is automatic, but our main results hold more generally for arbitrary eigenfunctions of $\cD_n$.)

The focus of this paper is on explicit estimates for various quantities related to $\phi$ that are central in the analytic theory of automorphic forms. We think of $n$ as fixed but potentially very large, and we emphasize that all estimates are \emph{uniform} as $\lambda_{\phi} \rightarrow \infty$. This includes
\begin{itemize}
\item uniform bounds for spherical Jacquet--Whittaker functions on $\GL_n(\RR)$;
\item uniform bounds for the ``essential support'' of $\phi$, i.e.\ the region outside which it decays exponentially;
\item uniform bounds for the global sup-norm of $\phi$, i.e.\ an upper bound for $\|\phi\|_{\infty}/\|\phi\|_2$ in terms of $\lambda_{\phi}$.
\end{itemize}
We proceed to discuss these points in more detail.

\subsection{Jacquet--Whittaker functions} The (standard archimedean spherical) Jacquet--Whittaker function $\cW_{\mu}$ on $\GL_n(\RR)$ associated with a cusp form $\phi$ is indexed by the (archimedean) Langlands parameters
 \begin{equation}\label{mueq4}
\mu=(\mu_1,\dots,\mu_n)\in\CC^n\qquad\text{with}\qquad\mu_1+\dots+\mu_n=0,
\end{equation}
so that in the tempered case (which by the generalized Ramanujan--Selberg conjecture should always be the case)
\begin{equation}\label{mueq1}
\mu=(\mu_1,\dots,\mu_n)\in(i\RR)^n\qquad\text{with}\qquad\mu_1+\dots+\mu_n=0.
\end{equation}
These parameters are only defined up to a permutation, and for convenience we order them to satisfy
\begin{equation}\label{mueq2}
\Im\mu_1\geq\dots\geq\Im\mu_n.
\end{equation}
In the non-tempered case we have the following weaker versions of \eqref{mueq1}:
\begin{equation}\label{mueq6}
\max\bigl(|\Re\mu_1|,\dots,|\Re\mu_n|\bigr)\leq\frac{1}{2}-\frac{1}{n^2+1},
\end{equation}
which is a celebrated result of Luo--Rudnick--Sarnak \cite[Thm.~1.2]{LRS}, and
\begin{equation}\label{mueq5}
\ov{\mu}=(\ov\mu_1,\ldots,\ov\mu_n)\qquad\text{is a permutation of}\qquad -\mu=(-\mu_1,\dots,-\mu_n),
\end{equation}
which reflects that the cuspidal representation of $\GL_n(\RR)$ generated by $\phi$ is unitary.

As the Laplace eigenvalue equals (cf.~\cite[Section~6]{M})
\[\lambda_\phi=\frac{n^3-n}{24}-\frac{\mu_1^2+\dots+\mu_n^2}{2},\]
it will be convenient for us to write
\begin{equation}\label{Tdef}
T_\mu:=\max(2,|\mu_1|,\dots,|\mu_n|)\asymp_n\lambda_\phi^{1/2}.
\end{equation}
We shall sometimes refer to Langlands parameters satisfying
\begin{equation}\label{generic}
|\mu_i-\mu_j|\gg T_\mu\qquad\text{for all}\qquad 1\leq i<j\leq n.
\end{equation}

The special function $\cW_\mu$ participates in the Fourier--Whittaker decomposition of $\phi$: it is invariant under $\Z_n(\RR)$, right-invariant under $\O_n(\RR)$, and transforms by a character under the left-action of $\U_n(\RR)$. Moreover, and crucially, $\cW_\mu$ is an eigenfunction of $\cD_n$ with the same eigenvalues as $\phi$. We view it as a center-invariant function on the positive diagonal torus with a particular $L^2$-normalization. For instance, in the tempered case \eqref{mueq1} we have
\begin{equation}\label{normalization}
\int_{(\RR_{>0})^{n-1}}\left|\cW_\mu(\diag(t_1,\dots,t_{n-1},1))\right|^2\,\prod_{j=1}^{n-1}\frac{dt_j}{t_j^{n+1-2j}}=
\frac{2^{1-n}\pi^{n/2}}{\Gamma(n/2)}.
\end{equation}
In the case $n=2$ the Jacquet--Whittaker function is essentially a $K$-Bessel function
\begin{equation}\label{WK}
\cW_{(\nu,-\nu)}(\diag(y,1))=\frac{2\pi^{1/2+\nu}\sqrt{y}K_{\nu}(2\pi y)}{\Gamma(1/2+\nu)}.
\end{equation}
In general we do not have explicit formulae for Jacquet--Whittaker functions, but due to Stade~\cite{S1} we can describe them recursively as iterated integrals of $K$-Bessel functions. This harmonizes with the fact that $\GL_n$ Kloosterman sums for the long Weyl element decompose into a product of (possibly degenerate) Kloosterman sums of smaller rank \cite[Cor.~3.11]{Stev}. Our first result captures, in a uniform fashion, the decay at zero and infinity of $\cW_\mu$.

\begin{theorem}\label{thm1} Let $t=\diag(t_1,\ldots,t_n)\in\GL_n(\RR)$ with $t_1,\dots,t_n>0$. Assume that the Langlands parameters satisfy \eqref{mueq1}--\eqref{mueq2}. Then for any $\eps>0$, we have
\begin{equation}\label{Wbound}
\cW_{\mu}(t)\ll_{n,\eps}C_{\mu,\eps}\left(\prod_{j=1}^n t_j^{n+1-2j}\right)^{1/2-\eps}
\exp\left(-\frac{1}{T_\mu}\sum_{j=1}^{n-1}\frac{t_j}{t_{j+1}}\right),
\end{equation}
where
\begin{equation}\label{Cdef}
C_{\mu,\eps}:=\prod_{1\leq j\leq n/2}|1+\mu_j-\mu_{n+1-j}|^{-(n+1-2j)/3+(n+1-2j)^2\eps}.
\end{equation}
In particular,
\begin{equation}\label{Wbound3}
\|\cW_{\mu}\|_{\infty}\ll_{n,\eps}C_{\mu,\eps}\,T_\mu^{(n^3-n)/12}.
\end{equation}
\end{theorem}

\begin{remark} The somewhat complicated definition \eqref{Cdef} reflects the subtle behavior of $\cW_{\mu}(t)$ in the various ranges. In particular, we cannot (completely) avoid the various coefficients of $\eps$ in \eqref{Cdef}, because that would invalidate \eqref{Wbound}. At any rate, $C_{\mu,\eps}\ll_{n,\eps}T_\mu^{-k(n-k)/3+kn^2\eps}$ holds whenever $1\leq k\leq n/2$ is an integer such that $|\mu_k-\mu_{n+1-k}|\gg T_\mu$, and here $k=1$ is always admissible, while under \eqref{generic} even $k=\lfloor n/2\rfloor$ is admissible.
\end{remark}

\begin{remark}
The bound \eqref{Wbound} improves substantially on \cite[Prop.~5.1]{BrT} in the present situation (note that $|W_\nu(a)|$ should be squared in that proposition). Moreover, \eqref{Wbound3} complements \cite[Thm.~1.4]{BrT}, which states under \eqref{mueq1} and \eqref{generic} the analogous lower bound $\|\cW_{\mu}\|_{\infty}\gg_n T_\mu^{n(n-1)(n-2)/12}$. The precise exponential decay of $\cW_{\mu}$ at infinity, but without uniformity in $\mu$, was obtained in \cite[Thm.~11.13]{KO}. Our method would yield the same if we used a stronger version of \eqref{Kbound} and \eqref{Kboundbis}, but the current formulation serves us better.
\end{remark}

\begin{remark} There are several conventions in the literature to parametrize positive diagonal matrices, such as \eqref{ydef} used in \cite{Bu,G,GH}, or \eqref{specialt} used in \cite{BrT,S1,S2,S3}. For clarity, we decided to display \eqref{Wbound} in terms of the matrix entries directly. We record that plugging either \eqref{ydef} or \eqref{specialt} for $t$, the bound \eqref{Wbound} would take the shape
\begin{equation}
\cW_{\mu}(t)\ll_{n,\eps} C_{\mu,\eps}\left(\prod_{i=1}^{n-1}y_i^{i(n-i)}\right)^{1/2-\eps}
\exp\left(-\frac{1}{T_\mu}\sum_{i=1}^{n-1}y_i\right).
\end{equation}
\end{remark}

For the sake of generality, we provide a variant of Theorem~\ref{thm1} valid for Langlands parameters potentially far away from the imaginary axis. It can be applied to studying non-tempered cusp forms (cf.\ \eqref{mueq6}), or analyzing Whittaker transforms with the help of Cauchy's theorem.

\begin{theorem}\label{thm1bis} Let $t=\diag(t_1,\ldots,t_n)\in\GL_n(\RR)$ with $t_1,\dots,t_n>0$. For given $\kappa>\delta>0$, assume that the Langlands parameters satisfy \eqref{mueq4}, \eqref{mueq2}, and
\begin{equation}\label{mueq3}
\max\bigl(|\Re\mu_1|,\dots,|\Re\mu_n|\bigr)\leq\kappa-\delta.
\end{equation}
Then we have
\begin{equation}\label{Wboundbis}
\cW_{\mu}(t)\ll_{n,\kappa,\delta}\tilde C_{\mu,\kappa}
\left(\prod_{j=1}^n t_j^{n+1-2j}\right)^{1/2-\kappa}\left|\prod_{j=1}^n t_j^{\mu_j+\mu_{n+1-j}}\right|^{1/2}
\exp\left(-\frac{1}{T_\mu}\sum_{j=1}^{n-1}\frac{t_j}{t_{j+1}}\right),
\end{equation}
where
\begin{equation}\label{Cdefbis}
\tilde C_{\mu,\kappa}:=\prod_{1\leq j\leq n/2}|2\kappa+\mu_j-\mu_{n+1-j}|^{(2n+1-4j)\kappa+(n+1-2j)^2\kappa}.
\end{equation}
\end{theorem}

\begin{remark}\label{remark4} We have $\tilde C_{\mu,\kappa}\ll_{n,\kappa}T_\mu^{(n-1)n(n+4)\kappa/6}$. The conclusion \eqref{Wboundbis} would fail for $\delta=0$, which also means that the implied constant blows up as $\delta\to 0+$. More precisely, for $\delta=0$, we would need to decrease the exponent $1/2-\kappa$ by an arbitrary $\eps>0$ and allow the implied constant to depend on $\eps>0$ (cf.\ \eqref{Wbound}), but this is just the same as the current formulation with $(\kappa+\eps,\eps)$ in place of $(\kappa,\delta)$. We prefer the current formulation for several reasons, e.g.\ because for $\kappa:=1/2$ and $\delta:=1/(n^2+1)$, the condition \eqref{mueq3} becomes \eqref{mueq6}.
\end{remark}

\subsection{Rapid decay in Siegel domains} The locally symmetric space $X_n$ has a fundamental domain lying in the standard Siegel set
\begin{equation}\label{siegel}
|x_{ij}|\leq 1/2\ \ \text{for}\ \ j>i\qquad\text{and}\qquad y_1,\dots,y_{n-1}\geq\sqrt{3}/2,
\end{equation}
with coordinates on $\cH_n:=\Z_n(\RR)\bs\GL_n(\RR)/\O_n(\RR)$ as in \cite[Def.~1.2.3]{G}. It is a well-known result,  attributed to several people including Gelfand, Piateskii-Shapiro, Harish-Chandra and Langlands, that $\phi(z)$ decays rapidly in Siegel sets. This has been generalized to much more general domains than Siegel sets \cite{MS,GMP}. Bernstein, in unpublished notes, strengthened this to \emph{exponential} decay, and this result was refined and  perfected by Kr\"otz and Opdam in \cite{KO}. On the other hand, none of these results is uniform in the Langlands parameters, and it is an interesting question how high in the cusp one needs to be to see the exponential decay. A precursor is given by the analytic behavior of the Jacquet--Whittaker function $\cW_\mu(z)$ considered in the previous subsection, which can blow up quite considerably, but eventually decays rapidly for $y_{n-j} = t_j/t_{j+1} \geq T_{\mu}$. Through the Fourier--Whittaker expansion, this should propagate to a quantitative exponential decay of $\phi(z)$ itself. Things are  more complicated, however, as $\U_n(\RR)$ only has a small abelian part to perform classical Fourier analysis, and therefore the Fourier--Whittaker expansion features the translates $\cW_\mu(\delta z)$ for certain matrices $\delta\in\M_n(\ZZ)$ with positive determinant. As $\det\delta\geq 1$, it is still possible to conclude that $\phi(z)$ decays rapidly as soon as the \emph{product} of the $y$-coordinates is sufficiently large:
\begin{equation}\label{product}
  \prod_{i=1}^{n-1} y_i\geq T_{\mu}^{n-1}.
\end{equation}
This aligns nicely with the situation of $\GL_2$ over a totally real number field, say, where the decay depends on the product of the $y$-coordinates in the various copies of the upper half plane. However, using tools from the geometry of numbers, we can say more.

\begin{theorem}\label{thm2} Let $\phi$ be an $L^2$-normalized Maa{\ss} cusp form on $X_n$, and let $z\in\cH_n$ be a point in the Siegel set \eqref{siegel}. There exists a constant $c_n>0$ such that
\begin{equation}\label{thm2bound}
\phi(z)\ll_n\lambda_{\phi}^{n^3}\exp\left(-c_n\,\cY(z)/T_{\mu}\right),
\end{equation}
where
\begin{equation}\label{cYdef}
\cY(z):=\max_{1\leq j\leq n-1}\max\left(\prod_{i=1}^j y_i^{j-i+1},\prod_{i=1}^j y_{n-i}^{j-i+1}\right)^{\frac{2}{j(j+1)}}.
\end{equation}
\end{theorem}

\begin{remark} We clearly have
\begin{equation}\label{cYbound}
\cY(z)\geq\left(\prod_{i=1}^{n-1} y_i^{n-i}\prod_{i=1}^{n-1}y_{n-i}^{n-i}\right)^{\frac{1}{(n-1)n}}
=\left(\prod_{i=1}^{n-1} y_i\right)^{\frac{1}{n-1}},
\end{equation}
which recovers the claim containing \eqref{product}, but the present result is stronger.
\end{remark}

\subsection{The global sup-norm} We now turn to a finer analysis of $\| \phi \|_{\infty}$. There has been enormous progress in recent years on the sup-norm problem for automorphic forms in various settings and with a focus on very different aspects such as: results valid for groups as general as possible \cite{BM,Ma}; bounds as strong as possible in terms of the exponent of $\lambda_{\phi}$ \cite{IS,BHM,BHMM}; results as uniform as possible in particular with respect to congruence covers of the underlying manifold \cite{T,Sah2,BHMM}; lower bounds for sup-norms \cite{Mi1,Mi2,Sah1,BrT,BrM}; results in the weight aspect for modular forms of integral or half-integral weight \cite{X,DS,Ki,FJK,Ste1,Ste2}; as well as results for certain types of Eisenstein series \cite{Bl1,HX}. Still, the literature on the sup-norm problem for groups of higher rank is fairly limited, and in particular for groups other than $\GL_2$ and $\GL_3$, there is no result available for the global sup-norm on non-compact quotients. The reason for this is related to the discussion of the previous two subsections. In small rank, the rapid decay of the Jacquet--Whittaker function and hence, to some extent, the rapid decay of the cusp form kicks in sufficiently early to make the analysis of compact and non-compact quotients fairly similar. In higher rank, the behavior changes completely, and it is the high peaks of the Jacquet--Whittaker function (of which the Airy-type bump of the $K$-Bessel function is a toy model) that dominate the sup-norm of a cusp form. This phenomenon was first observed by Brumley and Templier \cite{BrT} who used it to prove lower bounds on $\| \phi \|_{\infty}$ that can be much larger than standard bounds in any fixed compact part of the underlying space. Quantitatively, Sarnak's standard upper bound \cite{Sa} gives $\|\phi|_{\Omega}\|_\infty\ll_{n,\Omega}\lambda_{\phi}^{\gamma(n)}$ with $\gamma(n)<\dim X_n<n^2$ for any fixed compact subset $\Omega\subset X_n$, while Brumley and Templier obtained (for most cusp forms) a lower bound $\| \phi \|_{\infty} \gg \lambda_{\phi}^{\delta(n)}$ with  $\delta(n) \gg \height(\PGL_n) \gg n^3$. Here we prove a corresponding  upper bound for the global sup-norm of comparable order of magnitude.

\begin{theorem}\label{thm3} Let $\phi$ be an $L^2$-normalized Maa{\ss} cusp form on $X_n$, and let $z\in\cH_n$ be a point in the Siegel set \eqref{siegel}. Then  we have
\begin{equation}\label{thm3bound}
\phi(z)\ll_n\lambda_\phi^{(n^2-n)/8}+\lambda_\phi^{(n^2-n-1)/8}\prod_{i=1}^{n-1}y_i^{i(n-i)/2}.
\end{equation}
In particular, for any $\eps>0$, we have
\begin{equation}\label{thm4bound}
\|\phi\|_\infty\ll_{n,\eps}\lambda_\phi^{(n^2-2)(n+1)/16+\eps}.
\end{equation}
\end{theorem}

\begin{remark}
The bound \eqref{thm3bound} is proved by an application of Selberg's pre-trace formula and can be slightly improved for Hecke eigenforms by combining it with the amplification method. For $n$ of moderate size and in certain regions or for special forms, a careful investigation of the Fourier--Whittaker expansion can also lead to refined estimates, see Subsection~\ref{sub2} for more details. The upper bound \eqref{thm4bound} complements the lower bound $\|\phi\|_\infty\gg_{n,\eps}\lambda_\phi^{n(n-1)(n-2)/24-\eps}$ established under \eqref{mueq1} and \eqref{generic} by Brumley and Templier~\cite[Thm.~1.1]{BrT}. For $n=2$ and $n=3$, the best known upper bounds for $\|\phi\|_\infty$ can be found in \cite{IS} and \cite{BHM}, respectively.
\end{remark}

The paper is organized as follows. Section~\ref{section2} is devoted to Jacquet--Whittaker functions. Subsections~\ref{sub5}--\ref{sub6} contain background material, Subsection~\ref{sub7} contains the proof of Theorem~\ref{thm1}, and Subsection~\ref{sub8} contains the proof of Theorem~\ref{thm1bis}. Section~\ref{section3} is devoted to Maa{\ss} cusp forms. Subsections~\ref{sub1}--\ref{sub2} contain the proof of Theorem~\ref{thm2}, and Subsections~\ref{sub3}--\ref{sub4} contain the proof of Theorem~\ref{thm3}.

\begin{acknowledgement} We thank Antal Balog, Farrell Brumley, Jack Buttcane and Stephen D. Miller for useful discussions. We also thank the referee for reading the paper carefully and suggesting that we extend Theorem~\ref{thm1bis} to its current form.
\end{acknowledgement}

\section{Pointwise bounds for the Jacquet--Whittaker function}\label{section2}

Before proving Theorems~\ref{thm1} and \ref{thm1bis}, we collect first some basic facts about the Jacquet--Whittaker function for $\GL_n(\RR)$. Our references are Jacquet's seminal work \cite{J}, Stade's important series \cite{S1,S2,S3}, and selected chapters by Goldfeld~\cite[Ch.~5]{G} and Goldfeld--Hundley~\cite[Ch.~14]{GH}. We have also benefitted greatly from the excellent discussions of Brumley--Templier~\cite{BrT}, both in the original version and the current reduced version.

For the sake of discussion, we shall work with arbitrary parameters $\mu_j$ satisfying \eqref{mueq4}. The statement that these are the (archimedean) Langlands parameters of $\phi$ means, by definition, that the cuspidal representation of $\GL_n(\RR)$ generated by $\phi$ is isomorphic to the principal series representation parabolically induced from the character
\[t\mapsto\prod_{j=1}^n t_j^{\mu_j},\qquad t=\diag(t_1,\ldots,t_n),\qquad t_1,\ldots,t_n>0.\]
This representation is unitary, as reflected by the relation \eqref{mueq5}. In this section, we use this relation in Subsection~\ref{sub6} only. Until that time, our only assumption will be \eqref{mueq4}.

\subsection{Jacquet's functional equation}\label{sub5} By the Iwasawa decomposition, any matrix $g\in\GL_n(\RR)$ can be written as $g=utk$, where $u\in\U_n(\RR)$ is unipotent upper-triangular, $t=\diag(t_1,\ldots,t_n)$ is diagonal with positive diagonal entries, and $k\in\O_n(\RR)$ is orthogonal. Therefore, the height function
\[H_{\mu}(g):=\prod_{j=1}^n t_j^{(n+1)/2-j+\mu_j},\qquad g\in\GL_n(\RR),\]
is invariant under $\Z_n(\RR)$ and right-invariant under $\O_n(\RR)$, so it can be regarded as a function on $\PGL_n(\RR)$ and also as a function of $\cH_n$. We define the archimedean (spherical) Jacquet--Whittaker function by the formula
\begin{equation}\label{Wdef1}
\cW_{\mu}(g):=\int_{\U_n(\RR)} H_{\mu}(wug)\,\ov{\psi(u)}\,du,\qquad g\in\GL_n(\RR),
\end{equation}
where $w$ is the long Weyl-element, and
\begin{equation}\label{psidef}
\psi(u):=e(u_{1,2}+\dots+u_{n-1,n}),\qquad u=(u_{ij})\in\U_n(\RR),
\end{equation}
is the standard character of $\U_n(\RR)$. For any $g\in\GL_n(\RR)$, and for $\mu\in\CC^n$ lying in the positive Weyl chamber $\Re\mu_1>\dots>\Re\mu_n$, the integral in \eqref{Wdef1} converges locally uniformly and hence defines a holomorphic function $\mu\mapsto\cW_{\mu}(g)$. Jacquet~\cite{J} proved that this function extends holomorphically to every $\mu\in\CC^n$, and the completed Jacquet--Whittaker function
\begin{equation}\label{Wdef2}
W_{\mu}(g):=\left(\prod_{1\leq j<k\leq n}\Gamma_\RR(1+\mu_j-\mu_k)\right)\cW_{\mu}(g),\qquad g\in\GL_n(\RR),
\end{equation}
is invariant under any permutation of the $\mu_j$'s (action of the Weyl group). As usual, $\Gamma_\RR(s)$ abbreviates $\pi^{-s/2}\Gamma(s/2)$.

Implicit in Jacquet's paper \cite{J} is that not only $\mu\mapsto\cW_{\mu}(g)$ but even $\mu\mapsto W_{\mu}(g)$ is holomorphic on $\CC^n$ (e.g.\ the analogous adelic statement is \cite[Thm.~8.6]{J}). At any rate, an alternative and more direct proof of this stronger statement is provided by Stade's recursion, to be discussed in the next subsection. The proof of \cite[Thm.~3.4]{J} claims to apply a theorem by Hartogs on analytic continuation (probably a result from \cite{Ha}), but we believe it is really Bochner's tube theorem \cite{Bo} that is being used there.

\subsection{Stade's recursion} Stade~\cite{S1} discovered that the Jacquet--Whittaker function for $\GL_n(\RR)$ can be expressed as a certain $(n-2)$-dimensional integral involving the Jacquet--Whittaker function for $\GL_{n-2}(\RR)$. It generalizes the Vinogradov--Takhtadzhyan formula \cite{VT} that deals with the $n=3$ case. We shall quote Stade's result in the form of \cite[(4.3)]{S2}, because it fixes a constant from \cite[Thm.~2.1]{S1} and uses the Langlands parameters \eqref{mueq4} instead of the closely related spectral parameters
\begin{equation}\label{nueq}
\nu=(\nu_1,\dots,\nu_{n-1})\in\CC^{n-1}\qquad\text{given by}\qquad\nu_i:=(1+\mu_{n-i}-\mu_{n-i+1})/n.
\end{equation}

As $W_{\mu}$ is invariant under $\Z_n(\RR)$, right-invariant under $\O_n(\RR)$, and transforms by $\psi$ under the left-action of $\U_n(\RR)$, it suffices to understand its values at diagonal matrices of the form
\begin{equation}\label{specialt}
t=\diag(y_1y_2\dots y_{n-1}, y_2y_3\dots y_{n-1},\dots,y_{n-1},1),\qquad y_1,\dots,y_{n-1}>0.
\end{equation}
In accordance with \cite[(4.2)]{S2}, we introduce
\begin{equation}\label{Wdef3}
W_{\mu}^*(y_1,\dots,y_{n-1}):=W_{\mu}(t)\prod_{i=1}^{n-1}y_i^{-i(n-i)/2}\prod_{j=1}^i y_i^{-(\mu_j+\mu_{n+1-j})/2}.
\end{equation}
We note for later reference that if we write $t$ in \eqref{specialt} as $\diag(t_1,\dots,t_n)$, then $t_n=1$ and $y_i=t_i/t_{i+1}$, whence
\begin{equation}\label{convfactor}
\prod_{i=1}^{n-1}y_i^{-i(n-i)/2}\prod_{j=1}^i y_i^{-(\mu_j+\mu_{n+1-j})/2}=\left(\prod_{j=1}^n t_j^{(n+1-2j)+(\mu_j+\mu_{n+1-j})}\right)^{-1/2}.
\end{equation}

Now starting from \eqref{mueq4}, we introduce a shorter vector of Langlands parameters by dropping the first and the last entry of $\mu$, and then shifting the remaining $n-2$ entries to ensure that their sum is zero:
\[\mu':=\left(\mu_2+\frac{\mu_1+\mu_n}{n-2},\ldots,\mu_{n-1}+\frac{\mu_1+\mu_n}{n-2}\right)\in\CC^{n-2}.\]
With these notational conventions, Stade's recursion \cite[(4.3)]{S2} reads
\begin{equation}\label{recursion}\begin{split}
W_{\mu}^*(y_1,\dots,y_{n-1})=2^{2n-3}\int_{(\RR_{>0})^{n-2}}&\left\{\prod_{i=1}^{n-1}u_i^{\frac{\mu_{i+1}+\mu_{n-i}-\mu_1-\mu_n}{2}} K_{\frac{\mu_1-\mu_n}{2}}\left(2\pi y_i\sqrt{(1+u_{i-1}^2)(1+u_i^{-2})}\right)\right\}\\
&\times W_{\mu'}^*\left(y_2\frac{u_1}{u_2},\ldots,y_{n-2}\frac{u_{n-3}}{u_{n-2}}\right)\frac{du_1\dots du_{n-2}}{u_1\dots u_{n-2}},
\end{split}\end{equation}
where the implicit conventions $u_0=u_{n-1}^{-1}=0$ and $u_{n-1}^0=1$ are in place, and for $n=3$ the function $W_{\mu'}^*$ is understood to equal $1$.

Applying some crude bounds for the $K$-Bessel function (e.g.\ \cite[Prop.~9]{HM}), the integral in \eqref{recursion} converges locally uniformly for $\mu$ satisfying \eqref{mueq4}. This way we can see directly that $\mu\mapsto W_{\mu}(g)$ is holomorphic on $\CC^n$, and the reflection principle $W_{\ov\mu}=\ov{W_\mu}$ holds.

\subsection{Stade's formula}\label{sub6} Confirming a conjecture of Bump, Stade~\cite{S3} expressed the archimedean factor of certain Rankin--Selberg $L$-functions in terms of the archimedean Jacquet--Whittaker function. In our notation and in a special case, Stade's main result \cite[Thm.~1.1]{S3} reads
\[\int_{(\RR_{>0})^{n-1}}(W_\mu W_{-\mu})(\diag(t_1,\dots,t_{n-1},1))\,\prod_{i=1}^{n-1}\frac{t_i^{s-1}\,dt_i}{t_i^{n+1-2i}}=
\frac{2^{1-n}}{\Gamma_\RR(ns)}\prod_{j,k=1}^n\Gamma_\RR(s+\mu_j-\mu_k).\]
This formula is valid for any $\mu$ satisfying \eqref{mueq4} and for $\Re s$ sufficiently large in terms of $\mu$.

Using also the unitarity assumption \eqref{mueq5}, Jacquet's functional equation (invariance of $\mu\mapsto W_\mu$ under any permutation of the $\mu_j$'s), and the reflection principle (cf.\ previous subsection), we get $W_{-\mu}=W_{\ov\mu}=\ov{W_\mu}$, hence also $W_\mu W_{-\mu}=|W_\mu|^2$. That is, for $\mu$ satisfying \eqref{mueq4} and \eqref{mueq5}, Stade's formula yields
\begin{equation}\label{normalization2}
\int_{(\RR_{>0})^{n-1}}\left|W_\mu(\diag(t_1,\dots,t_{n-1},1))\right|^2\,\prod_{i=1}^{n-1}\frac{t_i^{s-1}\,dt_i}{t_i^{n+1-2i}}=
\frac{2^{1-n}}{\Gamma_\RR(ns)}\prod_{j,k=1}^n\Gamma_\RR(s+\mu_j-\mu_k).
\end{equation}

Finally, for tempered Langlands parameters as in \eqref{mueq1}, we infer by \eqref{Wdef2} the identity
\[\int_{(\RR_{>0})^{n-1}}\left|\cW_\mu(\diag(t_1,\dots,t_{n-1},1))\right|^2\,\prod_{i=1}^{n-1}\frac{t_i^{s-1}\,dt_i}{t_i^{n+1-2i}}=
\frac{2^{1-n}}{\Gamma_\RR(ns)}\prod_{j,k=1}^n\frac{\Gamma_\RR(s+\mu_j-\mu_k)}{\Gamma_\RR(1+\mu_j-\mu_k)}.\]
In the light of Theorem~\ref{thm1}, the last integral converges for $\Re s>0$, and evaluating it at $s=1$ yields the normalization \eqref{normalization} claimed in the Introduction. Note that \eqref{mueq2} can be assumed here without any loss of generality, because Jacquet's functional equation coupled with \eqref{mueq1} and \eqref{Wdef2} shows that $|\cW_\mu|$ is invariant under any permutation of the $\mu_j$'s.

\subsection{Proof of Theorem~\ref{thm1}}\label{sub7} The inequality \eqref{Wbound} is invariant under the action of the center $\Z_n(\RR)$, hence it suffices to prove it when $t_n=1$. Then we can parametrize $t$ as in \eqref{specialt}, with $y_i:=t_i/t_{i+1}$.
Using \eqref{WK}, we see that for $n=2$ the inequality \eqref{Wbound} is equivalent to the known bound (cf.\ \cite[p.~679]{BHo} and \cite[Prop.~9]{HM})
\begin{equation}\label{Kbound}
e^{\pi|\nu|/2} K_\nu(2\pi y)\ll_\eps |1+2\nu|^{-1/3+\eps}y^{-\eps}\exp\left(-\frac{y}{\max(2,|\nu|)}\right),\qquad\nu\in i\RR.
\end{equation}
Now we assume that either $n=3$, or $n\geq 4$ and \eqref{Wbound} holds for $n-2$ in place of $n$.

Using the definitions \eqref{Wdef2} and \eqref{Wdef3} along with the observation \eqref{convfactor}, we can rewrite \eqref{Wbound} as
\begin{equation}\label{rewrite}
\frac{W_{\mu}^*(y_1,\dots,y_{n-1})}{\prod_{1\leq j<k\leq n}\Gamma_\RR(1+\mu_j-\mu_k)}
\ll_{n,\eps} C_{\mu,\eps}\,\prod_{i=1}^{n-1} y_i^{-i(n-i)\eps}\exp\left(-\frac{y_i}{T_\mu}\right).
\end{equation}
We substitute the right hand side of \eqref{recursion} for the numerator, and we estimate the integrand in \eqref{recursion} by invoking \eqref{Kbound} with $(n-1)\eps$ in place of $\eps$ and the induction hypothesis \eqref{rewrite} with $\mu'$ in place of $\mu$. For the product of $K$-Bessel functions we obtain
\begin{equation}\label{productK}
\begin{split}
&\prod_{i=1}^{n-1}\frac{K_{\frac{\mu_1-\mu_n}{2}}\left(2\pi y_i\sqrt{(1+u_{i-1}^2)(1+u_i^{-2})}\right)}
{\Gamma_\RR(1+\mu_1-\mu_n)}\ll_{n,\eps}|1+\mu_1-\mu_n|^{-(n-1)/3+(n-1)^2\eps}\\
&\times\left\{\prod_{i=1}^{n-1} y_i^{-(n-1)\eps}\exp\left(-\frac{y_i}{T_\mu}\right)\right\}
\left\{\prod_{i=1}^{n-2}(u_i+u_i^{-1})^{-(n-1)\eps}\right\}.
\end{split}
\end{equation}
We observe that, by \eqref{mueq1}--\eqref{mueq2} and Stirling's formula,
\[\Gamma_\RR(1+\mu_1-\mu_n)\asymp
\Gamma_\RR(1+\mu_1-\mu_i)\Gamma_\RR(1+\mu_i-\mu_n),\qquad 2\leq i\leq n-1,\]
hence the product of gamma factors in \eqref{productK} is essentially of the same size as a specific subproduct of the gamma factors in \eqref{rewrite}:
\begin{equation}\label{productgamma}
\prod_{i=1}^{n-1}\Gamma_\RR(1+\mu_1-\mu_n)\asymp_n\prod_{\substack{{1\leq j<k\leq n}\\{\text{$j=1$ or $k=n$}}}}\Gamma_\RR(1+\mu_j-\mu_k).
\end{equation}
For the Jacquet--Whittaker function in the integral of \eqref{recursion}, we obtain for $n\geq 4$ (omitting the exponential factors)
\begin{equation}\label{inductionbound}
\frac{W_{\mu'}^*\left(y_2\frac{u_1}{u_2},\ldots,y_{n-2}\frac{u_{n-3}}{u_{n-2}}\right)}{\prod_{2\leq j<k\leq n-1} \Gamma_\RR(1+\mu_j-\mu_k)}
\ll_{n,\eps} C_{\mu',\eps}\left\{\prod_{i=2}^{n-2}y_i^{-(i-1)(n-1-i)\eps}\right\}
\left\{\prod_{i=1}^{n-2}u_i^{-(n-1-2i)\eps}\right\}.
\end{equation}
This bound is also valid for $n=3$, because in that case both sides are equal to $1$.

Combining \eqref{recursion} with \eqref{productK}--\eqref{inductionbound} and \eqref{Cdef}, we infer that
\begin{equation}\label{laststep}\begin{split}
&\frac{W_{\mu}^*(y_1,\dots,y_{n-1})}{\prod_{1\leq j<k\leq n} \Gamma_\RR(1+\mu_j-\mu_k)}\\
&\ll_{n,\eps}C_{\mu,\eps}\left\{\prod_{i=1}^{n-1} y_i^{-i(n-i)\eps}\exp\left(-\frac{y_i}{T_\mu}\right)\right\}
\int_{(\RR_{>0})^{n-2}}\prod_{i=1}^{n-2}(u_i^{n-1-i}+u_i^{-i})^{-2\eps}\;\frac{du_i}{u_i}.
\end{split}\end{equation}
The last integral splits and converges, hence \eqref{rewrite} follows as desired.

In order to prove \eqref{Wbound3}, consider an arbitrary positive diagonal matrix $t=\diag(t_1,\ldots,t_n)\in\GL_n(\RR)$, and denote by $T\in\{T_\mu,2T_\mu,3T_\mu,\dots\}$ the unique positive multiple of $T_\mu$ such that
\[T-T_\mu<\sum_{i=1}^{n-1}\frac{t_i}{t_{i+1}}\leq T.\]
Then clearly
\[\prod_{j=1}^n t_j^{n+1-2j}=\prod_{i=1}^{n-1}\left(\frac{t_i}{t_{i+1}}\right)^{i(n-i)}\leq\prod_{i=1}^{n-1}T^{i(n-i)}=T^{(n^3-n)/6},\]
and hence by \eqref{Wbound},
\[\cW_{\mu}(t)\ll_{n,\eps} C_{\mu,\eps}\,e^{-T/T_\mu}\,T^{(n^3-n)/12}\ll_n C_{\mu,\eps}\,T_\mu^{(n^3-n)/12},\]
with a bit to spare. The sup-norm bound \eqref{Wbound3} is immediate from here, since $\cW_{\mu}$ is right-invariant under $\O_n(\RR)$ and transforms by the character defined in \eqref{psidef} under the left-action of $\U_n(\RR)$.

\subsection{Proof of Theorem~\ref{thm1bis}}\label{sub8} We proceed similarly as in the proof of Theorem~\ref{thm1}, so we shall be brief. It suffices to prove \eqref{Wboundbis} for $t$ as in \eqref{specialt}, in which case it can be rewritten as (cf.\ \eqref{Wdef2} and \eqref{Wdef3}--\eqref{convfactor})
\begin{equation}\label{rewritebis}
\frac{W_{\mu}^*(y_1,\dots,y_{n-1})}{\prod_{1\leq j<k\leq n}\Gamma_\RR(1+\mu_j-\mu_k)}
\ll_{n,\kappa,\delta}\tilde C_{\mu,\kappa}\,\prod_{i=1}^{n-1} y_i^{-i(n-i)\kappa}\exp\left(-\frac{y_i}{T_\mu}\right).
\end{equation}
On the left hand side, we employ the bound (cf.\ \eqref{mueq2} and \eqref{mueq3})
\[\frac{1}{\Gamma_\RR(1+\mu_j-\mu_k)}\ll_{\kappa,\delta}e^{\pi\Im(\mu_j-\mu_k)/4}|2\kappa+\mu_j-\mu_k|^\kappa,\qquad 1\leq j<k\leq n,\]
which is obvious when $|\mu_j-\mu_k|\leq 1$ and a consequence of Stirling's approximation otherwise. By \eqref{mueq2} and \eqref{mueq3}, the product of the right hand side over all pairs $1\leq j<k\leq n$ is clearly
\[\ll_{n,\kappa,\delta}\prod_{1\leq j\leq n/2}e^{(n+1-2j)\pi\Im(\mu_j-\mu_{n+1-j})/4}|2\kappa+\mu_j-\mu_{n+1-j}|^{(2n+1-4j)\kappa},\]
hence for \eqref{rewritebis} it suffices to prove the slightly stronger inequality
\begin{equation}\label{rewritebis2}
W_{\mu}^*(y_1,\dots,y_{n-1})\prod_{1\leq j\leq n/2}e^{(n+1-2j)\pi\Im(\mu_j-\mu_{n+1-j})/4}
\ll_{n,\kappa,\delta}\hat C_{\mu,\kappa}\,\prod_{i=1}^{n-1} y_i^{-i(n-i)\kappa}\exp\left(-\frac{y_i}{T_\mu}\right),
\end{equation}
where (cf.\ \eqref{Cdefbis})
\[\hat C_{\mu,\kappa}:=\prod_{1\leq j\leq n/2}|2\kappa+\mu_j-\mu_{n+1-j}|^{(n+1-2j)^2\kappa}.\]

The new inequality \eqref{rewritebis2} is an analogue of \eqref{rewrite}, where the roles of $\eps$ and $C_{\mu,\eps}$
are played by $\kappa$ and $\hat C_{\mu,\kappa}$, respectively. For $n=2$, it is a consequence of \eqref{WK}, \eqref{Wdef2}, \eqref{Wdef3}--\eqref{convfactor} and the known bound (cf.\ \cite[Prop.~9]{HM})
\begin{equation}\label{Kboundbis}
e^{\pi|\Im\nu|/2} K_\nu(2\pi y)\ll_{n,\kappa,\delta}|\kappa+\nu|^{\sigma}y^{-\sigma}\exp\left(-\frac{y}{\max(2,|\nu|)}\right),
\qquad|\Re\nu|\leq\kappa-\delta,\qquad\kappa\leq\sigma\leq n\kappa.
\end{equation}
Now we assume that either $n=3$, or $n\geq 4$ and \eqref{rewritebis2} holds for $n-2$ in place of $n$.
Arguing as below \eqref{rewrite} with $\eps:=\kappa$, but using \eqref{Kboundbis} with $\sigma:=(n-1)\kappa$ instead of using \eqref{Kbound} with $(n-1)\eps$ in place of $\eps$, we arrive at the following variant of \eqref{laststep}:
\begin{equation}\label{laststepbis}\begin{split}
&W_{\mu}^*(y_1,\dots,y_{n-1})\prod_{1\leq j\leq n/2}e^{(n+1-2j)\pi\Im(\mu_j-\mu_{n+1-j})/4}\\
&\ll_{n,\kappa,\delta}\hat C_{\mu,\kappa}\left\{\prod_{i=1}^{n-1} y_i^{-i(n-i)\kappa}\exp\left(-\frac{y_i}{T_\mu}\right)\right\}
\int_{(\RR_{>0})^{n-2}}\prod_{i=1}^{n-2}(u_i+u_i^{-1})^{2\kappa-2\delta}(u_i^{n-1-i}+u_i^{-i})^{-2\kappa}\;\frac{du_i}{u_i}.
\end{split}\end{equation}
A new feature compared to \eqref{laststep} is the presence of $(u_i+u_i^{-1})^{2\kappa-2\delta}$ which bounds the factor $u_i^{(\mu_{i+1}+\mu_{n-i}-\mu_1-\mu_n)/2}$ in \eqref{recursion}. At any rate, the integral in \eqref{laststepbis} splits and converges, hence \eqref{rewritebis2} follows as desired.

\section{Pointwise bounds for Maa{\ss} cusp forms}\label{section3}

Let $\phi$ be an $L^2$-normalized Maa{\ss} cusp form on $X_n$ as in Theorems~\ref{thm2} and \ref{thm3}
with archimedean Langlands parameters $(\mu_1,\dots,\mu_n)$ ordered as in \eqref{mueq2}. Let $z=xy\in\cH_n$
be a point in the Siegel set \eqref{siegel}, where $x:=(x_{ij})\in\U_n(\RR)$ and
\begin{equation}\label{ydef}
y:=\diag(y_1y_2\dots y_{n-1},\dots,y_1y_2,y_1,1).
\end{equation}
It will be convenient for us to also write $y$ as $\diag(t_1,\dots,t_n)$, so that $t_n=1$ and $t_i/t_{i+1}=y_{n-i}$. For later reference, we record that the dual form
\begin{equation}\label{dualform}
\tilde\phi(z):=\phi\left((z^{-1})^t\right),\qquad z\in\cH_n,
\end{equation}
is an $L^2$-normalized Maa{\ss} cusp form on $X_n$ with (archimedean) Langlands parameters $(-\mu_n,\dots,-\mu_1)$, or alternatively $(\ov{\mu_n},\dots,\ov{\mu_1})$, ordered as in \eqref{mueq2} (cf.\ \eqref{mueq5}). To verify this, combine \cite[Prop.~9.2.1]{G} with \eqref{nueq}.

\subsection{Applying the Fourier--Whittaker expansion}\label{sub1} As a preparation for the proof of Theorem~\ref{thm2}, we examine first the special case when $\phi$ is a Hecke eigenform on $X_n$. As $\phi$ is an even Hecke--Maa{\ss} form (cf.\ \cite[Prop.~9.2.5~\&~9.2.6]{G}), we can and we shall renormalize it (i.e.\ scale it by a positive number) so that its Fourier--Whittaker expansion reads (cf.\ \cite[Thm.~9.3.11]{G})
\begin{equation}\label{arithmetic}
\phi(z)=\sum_\pm\sum_{m_1,\dots,m_{n-1}\geq 1}\frac{\lambda_{\phi}(m_1,\dots,m_{n-1})}{\prod_{i=1}^{n-1}m_i^{i(n-i)/2}}
\sum_{\gamma\in\U_{n-1}(\ZZ)\bs\SL_{n-1}(\ZZ)}\cW_\mu^\pm\left(m\begin{pmatrix}\gamma & \\ & 1\end{pmatrix}z\right),
\end{equation}
where $m$ abbreviates the diagonal matrix
\begin{equation}\label{mdef}
m:=\diag(m_1m_2\dots m_{n-1},\dots,m_1m_2,m_1,1),
\end{equation}
$\lambda_{\phi}(m_1,\dots,m_{n-1})\in\CC$ are the Hecke eigenvalues, and the functions $\cW_\mu^\pm$ have the same absolute value as $\cW_\mu$ considered earlier. In particular, $\lambda_{\phi}(1,\dots,1)=1$. Applying the Cauchy--Schwarz inequality, we obtain for any $\eps>0$,
\begin{equation}\label{phibound1}
\begin{split}
|\phi(z)|^2\ll & \left(\sum_{m_1,\dots,m_{n-1}\geq 1}\frac{|\lambda_{\phi}(m_1,\dots,m_{n-1})|^2}{\prod_{i=1}^{n-1}m_i^{(n-i)(1+\eps)}}\right) \\ \times&\left(\sum_{m_1,\dots,m_{n-1}\geq 1}
\frac{1}{\prod_{i=1}^{n-1}m_i^{(n-i)(i-1-\eps)}}\left(\sum_{\gamma\in\U_{n-1}(\ZZ)\bs\SL_{n-1}(\ZZ)}\left|\cW_\mu\left(m\begin{pmatrix}\gamma & \\ & 1\end{pmatrix}z\right)\right|\right)^2\right).
\end{split}
\end{equation}
Using \cite[Def.~12.1.2]{G} and \cite[Thm.~2]{L}, we can bound the first factor as
\begin{equation}\label{RSbound}
\sum_{m_1,\dots,m_{n-1}\geq 1}\frac{|\lambda_{\phi}(m_1,\dots,m_{n-1})|^2}{\prod_{i=1}^{n-1}m_i^{(n-i)(1+\eps)}}
=\frac{L(1+\eps,\phi\times\tilde\phi)}{\zeta(n+n\eps)}\ll_{n,\eps} T_\mu^\eps,
\end{equation}
so we focus on the second factor.

The size of the Jacquet--Whittaker function in \eqref{phibound1} depends on $m$ and the diagonal Iwasawa coordinates of $\begin{pmatrix}\gamma & \\ & 1\end{pmatrix}z$. In order to control this size and also the number of $\gamma$'s corresponding to a given $m$ and a given size range, we denote by $z'$ the upper left $(n-1)\times(n-1)$ block of $z$ and record the Iwasawa decomposition
\begin{equation}\label{iwasawa}
\gamma z'=usk\qquad\text{with}\qquad u\in\U_{n-1}(\RR),\ \ \text{$s=\diag(s_1,\dots,s_{n-1})$ and $s_i>0$},\ \ k\in\O_{n-1}(\RR).
\end{equation}
For notational convenience, we also set
\begin{equation}\label{sn}
s_n:=1\qquad\text{and}\qquad S_n:=1.
\end{equation}
We can and we shall choose representatives $\gamma\in\SL_{n-1}(\ZZ)$ such that the entries of the unipotent part $u=(u_{ij})$ satisfy $|u_{ij}|\leq 1/2$ for $j>i$. Then in \eqref{phibound1} we have the Iwasawa decomposition
\[\begin{pmatrix}\gamma & \\ & 1\end{pmatrix}z=\begin{pmatrix}\gamma z' & \ast \\ & 1 \end{pmatrix}=
\begin{pmatrix} u & \ast \\ & 1\end{pmatrix}\begin{pmatrix} s & \\ & 1 \end{pmatrix}\begin{pmatrix} k & \\ & 1 \end{pmatrix},\]
hence by \eqref{Wbound} we obtain in the tempered case \eqref{mueq1}
\begin{equation}\label{Wbound2}
\cW_\mu\left(m\begin{pmatrix}\gamma & \\ & 1\end{pmatrix}z\right)
\ll_{n,\eps}C_{\mu,\eps}\left(\prod_{i=1}^{n-1}m_{n-i}^{i(n-i)}s_i^{n+1-2i}\right)^{1/2-\eps}
\exp\left(-\frac{1}{T_\mu}\sum_{i=1}^{n-1}m_{n-i}\frac{s_i}{s_{i+1}}\right),
\end{equation}
and by \eqref{Wboundbis} we obtain in the non-tempered case \eqref{mueq6}--\eqref{mueq5}
\begin{equation}\label{Wbound2bis}
\cW_\mu\left(m\begin{pmatrix}\gamma & \\ & 1\end{pmatrix}z\right)
\ll_n\tilde C_{\mu,1/2}\,\left|\prod_{i=1}^{n-1}(m_1\dots m_{n-i} s_i)^{\mu_i+\mu_{n+1-i}}\right|^{1/2}
\exp\left(-\frac{1}{T_\mu}\sum_{i=1}^{n-1}m_{n-i}\frac{s_i}{s_{i+1}}\right).
\end{equation}

This motivates the following definition. For an arbitrary $T\in\{T_\mu,2T_\mu,3T_\mu,\dots\}$ (a positive multiple of $T_\mu$), a diagonal matrix $m$ as in \eqref{mdef}, and a diagonal matrix
\[S:=\diag(S_1,\dots,S_{n-1})\]
whose entries are powers of $2$ with integer exponents (including negative integer exponents), we denote by
$\cB(T,m,S)$ the set of $\gamma\in\SL_{n-1}(\ZZ)$ such that in \eqref{iwasawa} we have (with the convention \eqref{sn})
\begin{equation}\label{block1}
|u_{ij}|\leq 1/2\ \ \text{for}\ \ j>i,\qquad T-T_\mu<\sum_{i=1}^{n-1}m_{n-i}\frac{s_i}{s_{i+1}}\leq T,\qquad S_i/2<s_i\leq S_i.
\end{equation}
By \eqref{siegel} and \eqref{Sproduct} below, $\cB(T,m,S)$ is empty unless $S_{n-1}S_{n-2}\dots S_i\gg_n 1$ holds for all $i$, hence we shall impose this restriction from now on. The bounds in \eqref{block1} now imply
\begin{equation}\label{block2}
\frac{S_i}{S_{i+1}}<\frac{2s_i}{s_{i+1}}\leq\frac{2T}{m_{n-i}}\qquad\text{and}\qquad T^{-(n-i)(n-i-1)/2}\ll_n S_i\ll_n \frac{T^{n-i}}{m_1\dots m_{n-i}},
\end{equation}
so that in particular $m_1\dots m_{n-1}\ll_n T^{n(n-1)/2}$. It follows that, for a given $T\in\{T_\mu,2T_\mu,3T_\mu,\dots\}$, the number of relevant dyadic diagonal matrices $S$ is $\ll_{n,\eps}T^\eps$, while for any $\gamma\in\cB(T,m,S)$ we have in \eqref{Wbound2}
\begin{equation}\label{inWbound2}
\prod_{i=1}^{n-1}m_{n-i}^{i(n-i)}s_i^{n+1-2i}=\prod_{i=1}^{n-1}\left(m_{n-i}\frac{s_i}{s_{i+1}}\right)^{i(n-i)}
<\prod_{i=1}^{n-1}T^{i(n-i)}=T^{(n^3-n)/6},
\end{equation}
and similarly we have in \eqref{Wbound2bis}
\begin{equation}\label{inWbound2bis}
\prod_{i=1}^{n-1}(m_1\dots m_{n-i} s_i)^{\mu_i+\mu_{n+1-i}}\ll_n\prod_{i=1}^{n-1}T^{(n-i)(n-i+1)/2}=T^{(n^3-n)/6}.
\end{equation}

By \eqref{phibound1}--\eqref{RSbound}, \eqref{Wbound2}--\eqref{Wbound2bis}, \eqref{inWbound2}--\eqref{inWbound2bis} we infer
\begin{equation}\label{phibound2}
|\phi(z)|^2\ll_{n,\eps}C^2\, \sum_{m_1,\ldots,m_{n-1}\geq 1} \frac{1}{\prod_{i=1}^n m_i^{(n-i)(i-1-\eps)}} \left(\sum_T e^{-T/T_{\mu}} T^{(n^3-n)/12+\eps} \max_S \#\cB(T,m,S) \right)^2,
\end{equation}
where $C$ denotes $C_{\mu,\eps}$ or $\tilde C_{\mu,1/2}$ depending on whether we are in the tempered case \eqref{mueq1} or in the non-tempered case \eqref{mueq6}--\eqref{mueq5}. In the next subsection, we shall estimate $\#\cB(T,m,S)$ by the geometry of numbers.

\subsection{Geometry of numbers} Let us consider the lattice $\Lambda\subset\RR^{n-1}$ spanned by the rows of $z'$. By \eqref{siegel}, the rows of $z'$ constitute a reduced basis of $\Lambda$ in the sense of (1.4)--(1.5) in \cite{LLL}, hence by (1.7) and (1.12) in the same paper, the $i$-th successive minimum $\lambda_i$ of $\Lambda$ is of size
\begin{equation}\label{successive}
\lambda_i\asymp_n t_{n-i}=y_1\dots y_i,\qquad 1\leq i\leq n-1.
\end{equation}
See also the Remark after \cite[(1.13)]{LLL} for a related comment. Now, for a given $\gamma\in\cB(T,m,S)$, the rows of $\gamma z'$ constitute an alternative basis of $\Lambda$. We can localize these rows recursively in terms of $S$, by combining \eqref{iwasawa} with the first and last part of \eqref{block1}. Indeed, if $\gamma_i$ (resp.\ $v_i$) denotes the $i$-th row of $\gamma$ (resp.\ $sk$), then $u$ is the coordinate matrix of the basis $(\gamma_1z',\dots\gamma_{n-1}z')$ when expressed in the orthogonal basis $(v_1,\dots,v_{n-1})$, and in the latter basis $v_i$ has length $s_i$. For each $i$, the tail $(\gamma_i z',\dots,\gamma_{n-1}z')$ generates an $(n-i)$-dimensional sublattice of $\Lambda$ with covolume $s_{n-1}s_{n-2}\dots s_i$ and successive minima \emph{at least} the corresponding minima of $\Lambda$, hence combining \eqref{block1}, \eqref{successive}, and a theorem of Minkowski~\cite[Thm.~3 on p.~124]{GL}, we infer
\begin{equation}\label{Sproduct}
S_{n-1}S_{n-2}\dots S_i\gg_n t_{n-1}t_{n-2}\dots t_i=y_1^{n-i}y_2^{n-i-1}\dots y_{n-i}.
\end{equation}
Moreover, for each $i$, we have an orthogonal decomposition $\gamma_i z'=v_i+\sum_{j>i}u_{ij}v_j$, hence $v_{n-1}$, $v_{n-2}$, etc. can be obtained recursively (in this order) by a Gram--Schmidt process from $\gamma_{n-1}z'$, $\gamma_{n-2}z'$, etc. In particular, the tail $(\gamma_{i+1}z',\dots,\gamma_{n-1}z')$ determines $(v_{i+1},\dots,v_{n-1})$, and the lattice vector $\gamma_i z'\in\Lambda$ lies in the following orthogonal Minkowski sum depending only on $S_i$ and the tail:
\[\left\{v\in\langle v_{i+1},\dots, v_{n-1}\rangle^\perp:\ \|v\|\leq S_i\right\}+[-1,1]v_{i+1}+\cdots+[-1,1]v_{n-1}.\]
For convenience, we cover this convex body by the Minkowski sum of $n-1$ pairwise orthogonal line segments centered at the origin with radii (half-lengths)
\[\overbrace{S_i, S_i, \dots, S_i}^\text{$i$ times}, S_{i+1}, \dots, S_{n-1}.\]

We can now bound the number of possible $\gamma$'s by bounding $\#\gamma_{n-1}$, $\#\gamma_{n-2}$, etc., $\#\gamma_2$, in this order, keeping in mind that the first row $\gamma_1$ is uniquely determined by the other rows (as follows from the first part of \eqref{block1}). Combining Lemma~\ref{lemma1} below with \eqref{successive}, we see that for $2\leq i\leq n-1$ and for given $(\gamma_{i+1},\dots,\gamma_{n-1})$, the number of possibilities for $\gamma_i$ is
\begin{equation}\label{gammaibound}
\#\gamma_i\ll_n\sum_{n-i\leq d\leq n-1}
\frac{\sigma_d(S_i, S_i, \dots, S_i, S_{i+1}, \dots, S_{n-1})}{t_{n-1}\dots t_{n-d}},
\end{equation}
where $\sigma_d$ is the $d$-th symmetric polynomial in $n-1$ variables:
\[\sigma_d(X_1,\dots,X_{n-1}):=\sum_{1\leq i_1<\dots<i_d\leq n-1}X_{i_1}\dots X_{i_d}.\]
Estimating the terms in \eqref{gammaibound} somewhat crudely via \eqref{siegel} and \eqref{block2}, we obtain
\[\#\gamma_i\ll_n\frac{\sigma_{n-1}(T^{n-i}, T^{n-i}, \dots, T^{n-i}, T^{n-i-1}, \dots, T)}{m_1^{n-i}t_{n-1}\dots t_i}=
\frac{T^{(n^2-n-i^2+i)/2}}{m_1^{n-i}t_{n-1}\dots t_i},\]
and finally
\begin{equation}\label{blockbound}
\#\cB(T,m,S)\ll_n\prod_{i=2}^{n-1}\frac{T^{(n^2-n-i^2+i)/2}}{m_1^{n-i}t_{n-1}\dots t_i}=
\frac{T^{n(n-1)(n-2)/3}}{m_1^{(n-1)(n-2)/2}\prod_{j=2}^{n-1}t_j^{j-1}}.
\end{equation}
We stress that the left hand side vanishes unless \eqref{block2} and \eqref{Sproduct} hold for all $i$.

We end this subsection with a general lemma concerning the number of lattice points contained in an orthotope: a parallelotope whose edges are pairwise orthogonal. It is rather standard, but the way we formulate it and prove it might be of some interest.

\begin{lemma}\label{lemma1} Let $\Lambda\subset\RR^m$ be a lattice with successive minima $\lambda_1\leq\dots\leq\lambda_m$. Let $K\subset\RR^m$ be an orthotope symmetric about the origin. Assume that the linear span of $\Lambda\cap K$ has dimension $1\leq d\leq m$, and denote by $V_d(K)$ the maximum of the $d$-volumes of the $d$-dimensional faces of $K$. Then we have
\[\#(\Lambda\cap K)\ll_d\frac{V_d(K)}{\lambda_1\dots\lambda_d}.\]
\end{lemma}

\begin{proof} Without loss of generality, the edges of $K$ are parallel to the standard coordinate axes of $\RR^m$. Let $W$ denote the $d$-dimensional subspace spanned by $\Lambda\cap K$. In this subspace, consider the lattice $\Lambda':=\Lambda\cap W$ and the convex body $K':=K\cap W$. Then $\Lambda'$ has a fundamental parallelotope $P$ lying in $d^2 K'$ by an observation of Mahler~\cite[p.~68]{GL}, and its $d$-volume satisfies
$\vol_d(P)\gg_d\lambda_1\dots\lambda_d$ by a theorem of Minkowski~\cite[Thm.~3 on p.~124]{GL}. The translates of $P$ by the elements of $\Lambda'\cap K'$ are pairwise disjoint and lie in $(d^2+1)K'$, therefore
\[\#(\Lambda\cap K)=\#(\Lambda'\cap K')\leq\frac{\vol_d((d^2+1)K')}{\vol_d(P)}\ll_d\frac{\vol_d(K')}{\lambda_1\dots\lambda_d}.\]
By a generalized Pythagorean theorem (see e.g.\ \cite{CB}), the square of $\vol_d(K')$ equals the sum of the squares of the $d$-volumes of the $\binom{m}{d}$ orthogonal projections of $K'$ to the $d$-dimensional coordinate subspaces of $\RR^m$. Hence $\vol_d(K')\ll_d V_d(K)$, and the result follows.
\end{proof}

\subsection{Concluding Theorem~\ref{thm2}}\label{sub2} We are still examining the special case when $\phi$ is a Hecke--Maa{\ss} cusp form on $X_n$, renormalized to satisfy \eqref{arithmetic}. The general case of an arbitrary $L^2$-normalized Maa{\ss} cusp form will be reduced to this special case at the end of this subsection.

Let us first assume that we are in the tempered case \eqref{mueq1}. By \eqref{block2} and \eqref{Sproduct}, we can restrict in \eqref{phibound2} the sum over $T\in\{T_\mu,2T_\mu,3T_\mu,\dots\}$ to $T\geq 3c_n T(z)$, where $c_n>0$ is a suitable constant, and
\[T(z):=\max_{1\leq j\leq n-1}\left(\prod_{i=1}^j y_i^{j-i+1}\right)^{\frac{2}{j(j+1)}}.\]
For the same reason, we can further restrict to $T$ satisfying $T^{n(n-1)/2}\gg_n m_1\ldots m_{n-1}$, and insert in the $T$-sum the redundant factor
\[\left(\frac{T^{n(n-1)/2}}{m_1\ldots m_{n-1}}\right)^{2n\eps} \gg_{n,\eps} 1.\]
Then, invoking \eqref{blockbound}, we obtain the uniform bound
\begin{align}\label{phibound3}
\phi(z)&\ll_{n,\eps} C_{\mu,\eps}\,\frac{T(z)^{n(n-1)(5n-7)/12+\eps}}{\prod_{i=1}^{n-1}y_i^{(n-i)(n-i-1)/2}} \exp\left(-3c_n\frac{T(z)}{T_{\mu}}\right),\qquad n\geq 3,\notag\\
\intertext{hence also}
\phi(z)&\ll_{n,\eps} C_{\mu,\eps}\,\frac{T_\mu^{n(n-1)(5n-7)/12+\eps}}{\prod_{i=1}^{n-1}y_i^{(n-i)(n-i-1)/2}} \exp\left(-2c_n\frac{T(z)}{T_{\mu}}\right)
,\qquad n\geq 3.
\end{align}
The restriction to $n\geq 3$ guarantees that the resulting sum over $m_1,\dots,m_{n-1}\geq 1$ converges.
In the case of $n=2$, the exponent of $m_1$ in \eqref{blockbound} vanishes, hence we insert an additional factor of $(T/m_1)^{1/2}\gg 1$ to achieve convergence, and we conclude that the above bound holds with an additional factor of $T_\mu^{1/2}$.

We can derive an alternative version of \eqref{phibound3} with the help of the dual form introduced in \eqref{dualform}. We express $\phi(z)$ as $\tilde\phi(\tilde z)$, where
\[\tilde z:=(y_1\dots y_{n-1})\cdot w(z^{-1})^t w\]
has the Iwasawa decomposition (cf.\ \cite[(9.2.2)]{G})
\[\tilde z=\tilde x\tilde y\ \ \text{with}\ \ \tilde x\in\U_n(\RR)\ \ \text{and}\ \
\tilde y:=\diag(y_{n-1}y_{n-2}\dots y_1,\dots,y_{n-1}y_{n-2},y_{n-1},1).\]
Applying \eqref{phibound3} to $\tilde\phi$ (resp.\ $\tilde z$) in place of $\phi$ (resp.\ $z$), we infer
\begin{equation}\label{phibound4}
\phi(z)\ll_{n,\eps}C_{\mu,\eps}\,\frac{T_\mu^{n(n-1)(5n-7)/12+\eps}}{\prod_{i=1}^{n-1}y_i^{i(i-1)/2}}
\exp\left(-2c_n\frac{T(\tilde z)}{T_{\mu}}\right),\qquad n\geq 3,
\end{equation}
and the same for $n=2$ with an additional factor of $T_\mu^{1/2}$.

The quantities on the left hand sides of \eqref{phibound3}--\eqref{phibound4} are the same, hence multiplying the two inequalities and then taking the square root, we obtain by \eqref{Tdef} and \eqref{cYdef},
\begin{equation}\label{phibound6}
\phi(z)\ll_{n,\eps}C_{\mu,\eps}\,\frac{\lambda_\phi^{n(n-1)(5n-7)/24+\eps}}{\prod_{i=1}^{n-1}y_i^{n(n-1)/4-i(n-i)/2}}
\exp\left(-c_n\frac{\cY(z)}{T_{\mu}}\right),\qquad n\geq 3,
\end{equation}
and the same for $n=2$ with an additional factor of $\lambda_\phi^{1/4}$. Combining Rankin--Selberg theory with \eqref{normalization} and Brumley's lower bound \cite[Thm.~3]{Br} (see also \cite[Appendix]{La}), we see that
\[\|\phi\|_2^2\asymp_n L(1,\phi,\Ad)\gg_{n,\eps}\lambda_\phi^{-(n-1)^2/4-\eps},\]
hence we conclude
\[\frac{\phi(z)}{\|\phi\|_2}\ll_{n,\eps}C_{\mu,\eps}\,\frac{\lambda_\phi^{(n-1)(5n^2-4n-3)/24+\eps}}{\prod_{i=1}^{n-1}y_i^{n(n-1)/4-i(n-i)/2}}
\exp\left(-c_n\frac{\cY(z)}{T_{\mu}}\right),\qquad n\geq 3.\]
In the case of $n=2$, the same holds with an additional factor of $\lambda_\phi^{1/8}$ (instead of $\lambda_\phi^{1/4}$), because in this case we have $L(1,\phi,\Ad)\gg_\eps\lambda_\phi^{-\eps}$ by the celebrated result of Hoffstein and Lockhart~\cite[Thm.~0.2]{HL}.

We have similar but slightly weaker results in the non-tempered case \eqref{mueq6}--\eqref{mueq5}. As before, we obtain \eqref{phibound6} but with $\tilde C_{\mu,1/2}$ in place of $C_{\mu,\eps}$. In connection with Rankin--Selberg theory, we use the following extension of \eqref{normalization}, which is a consequence of \eqref{Wdef2} and \eqref{normalization2}:
\[\int_{(\RR_{>0})^{n-1}}\left|\cW_\mu(\diag(t_1,\dots,t_{n-1},1))\right|^2\,\prod_{i=1}^{n-1}\frac{dt_i}{t_i^{n+1-2i}}=
\frac{2^{1-n}}{\Gamma_\RR(n)}\prod_{j<k}\frac{\Gamma_\RR(1-\mu_j+\mu_k)}{\Gamma_\RR(1+\ov{\mu_j}-\ov{\mu_k})}.\]
Then, by Brumley's lower bound \cite[Thm.~3]{Br} (see also \cite[Appendix]{La}) and Stirling's formula, we obtain
\[\|\phi\|_2^2\asymp_n L(1,\phi,\Ad)\prod_{j<k}\frac{\Gamma_\RR(1-\mu_j+\mu_k)}{\Gamma_\RR(1+\ov{\mu_j}-\ov{\mu_k})}
\gg_{n,\eps}\lambda_\phi^{-(n-1)^2/4-\eps}\lambda_\phi^{-n(n-1)/4},\]
hence we conclude
\begin{equation}\label{phibound7}
\frac{\phi(z)}{\|\phi\|_2}\ll_{n,\eps}\tilde C_{\mu,1/2}\,\frac{\lambda_\phi^{(n-1)(5n^2-n-3)/24+\eps}}{\prod_{i=1}^{n-1}y_i^{n(n-1)/4-i(n-i)/2}}
\exp\left(-c_n\frac{\cY(z)}{T_{\mu}}\right),\qquad n\geq 3.
\end{equation}
In the case of $n=2$, we are automatically in the tempered case \eqref{mueq1}, so our earlier discussion applies.

Finally, let $\phi$ be an arbitrary $L^2$-normalized Maa{\ss} cusp form on $X_n$ as in Theorem~\ref{thm2}. Decomposing $\phi$ into pairwise orthogonal Hecke eigenforms, we see that $\phi(z)$ can be bounded by the right hand side of \eqref{phibound7} multiplied by the \emph{square-root} of the multiplicity of the Laplace eigenvalue $\lambda_\phi$. The multiplicity is $\ll_n\lambda_\phi^{(n+2)(n-1)/4}$
by Donnelly's theorem \cite[Thm.~1.1]{Do}, while $\tilde C_{\mu,1/2}\ll_n\lambda_\phi^{(n-1)n(n+4)/24}$ by Remark~\ref{remark4}, hence in the end we obtain (for $n=2$ as well)
\[\phi(z)\ll_n\frac{\lambda_\phi^{n^3/4}}{\prod_{i=1}^{n-1}y_i^{n(n-1)/4-i(n-i)/2}}
\exp\left(-c_n\frac{\cY(z)}{T_{\mu}}\right).\]
This is clearly stronger than \eqref{thm2bound}, hence the proof of Theorem~\ref{thm2} is complete.

\subsection{Applying Selberg's pre-trace formula}\label{sub3} Let $\phi$ be an $L^2$-normalized Maa{\ss} cusp form on $X_n$ as in Theorem~\ref{thm3}, and let $z\in\cH_n$
be a point in the Siegel set \eqref{siegel}. Note that the Langlands parameters \eqref{mueq4} are not necessarily purely imaginary, but they satisfy \eqref{mueq6}--\eqref{mueq5}. We shall establish \eqref{thm3bound} with the help of Selberg's pre-trace formula.  As a preparation, we denote by
\[\mathfrak{a}:=\{\diag(\alpha_1,\dots,\alpha_n)\in\M_n(\RR):\ \alpha_1+\dots+\alpha_n=0\}\]
the Lie algebra of the diagonal torus of $\PGL_n(\RR)$, and by
$C:\PGL_n(\RR)\to\mathfrak{a}/S_n$ the Cartan projection induced from the Cartan decomposition for $\GL_n(\RR)$:
\[g=h k_1\exp (C(g)) k_2\qquad\text{with}\qquad h\in\Z_n(\RR),\ k_1,k_2\in\O_n(\RR).\]
We identify the complexified dual $\mathfrak{a}^*_{\CC}$ with the set of vectors (cf.\ \eqref{mueq4})
\[\kappa=(\kappa_1,\dots,\kappa_n)\in\CC^n\qquad\text{satisfying}\qquad\kappa_1+\dots+\kappa_n=0,\]
namely such a vector acts on $\mathfrak{a}_{\CC}:=\mathfrak{a}\otimes_\RR\CC$ by the $\CC$-linear map
\[\kappa:\diag(\alpha_1,\dots,\alpha_n)\mapsto \kappa_1\alpha_1+\dots+\kappa_n\alpha_n.\]

Following \cite[Section~2]{BM} and \cite[Sections~2~\&~6]{BP}, we can construct a smooth, bi-$\O_n(\RR)$-invariant function
$f_{\mu}:\PGL_n(\RR)\to\CC$ supported in a fixed compact subset $\Omega$ (which is independent of $\mu\in\mathfrak{a}^*_{\CC}$) with the following properties. On the one hand, the function obeys (cf.\ \cite[(2.3)--(2.4)]{BM})
\begin{equation}\label{eq:inverse_spherical_transform_BP_bound}
f_{\mu}(g)\ll \lambda_{\phi}^{(n^2-n)/4}\left(1+\lambda_{\phi}^{1/2}\|C(g)\|\right)^{-1/2},\qquad g\in\PGL_n(\RR),
\end{equation}
where $\|\cdot\|$ stands for a fixed $S_n$-invariant norm on $\mathfrak{a}\cong\RR^{n-1}$. On the other hand, its spherical transform
$\tilde{f}_{\mu}:\mathfrak{a}^*_{\CC}/S_n \to\CC$, defined as in \cite[(17)~of~Section~II.3]{He} or \cite[(2.3)]{BP}, satisfies
\[\tilde{f}_{\mu}(\mu)\geq 1,\qquad \tilde{f}_{\mu}(\kappa)\geq 0\]
for all Langlands parameters $\kappa\in\mathfrak{a}^*_{\CC}$ occurring in $L^2(X_n)$, including possibly non-tempered parameters. Then, using positivity in Selberg's pre-trace formula (see \cite[(6.1)]{BP}), or more directly by a Mercer-type pre-trace inequality (cf.\ \cite[(3.15)]{BHMM}), we obtain
\begin{equation}\label{eq:pre-trace_inequality}
|\phi(z)|^2 \leq \sum_{\gamma\in\PGL_n(\ZZ)} f_\mu(z^{-1}\gamma z).
\end{equation}

\subsection{Counting matrices} In this subsection, all implied constants depend on $n$, $\Omega$, and for simplicity we do not indicate this dependence. We are interested in estimating the right hand side of \eqref{eq:pre-trace_inequality}, hence we restrict to $\gamma\in\PGL_n(\ZZ)$ such that $z^{-1}\gamma z\in\Omega$. For any such $\gamma$ (more precisely for any representative $\gamma\in\GL_n(\ZZ)$), we denote by $\gamma_{ij}$ (resp.\ $\Gamma_{ij}$) the entries of $\gamma$ (resp.\ $z^{-1}\gamma z$), so that $\gamma=(\gamma_{ij})$ and $z^{-1}\gamma z=(\Gamma_{ij})$. Then, we can generalize \cite[(33)]{BHM} as follows.

\begin{lemma}\label{lemma: entries of z(-1) gamma z} The entries of $z^{-1}\gamma z$ equal
\[\Gamma_{ij}:=(z^{-1}\gamma z)_{ij} = Y_{ij} \sum_{r=i}^n\sum_{s=1}^j c_{rs}^{(ij)}(x) \gamma_{rs},\]
where
\begin{equation}\label{Ydef}
Y_{ij}:=\begin{cases} y_{n-i+1}\dots y_{n-j},\qquad & \text{if $i>j$;}\\
(y_{n-j+1}\dots y_{n-i})^{-1},\qquad & \text{if $i<j$;}\\
1,\qquad & \text{if $i=j$.}\end{cases}\end{equation}
In this identity, the coefficients $c_{rs}^{(ij)}(x)$ only depend on the upper-triangular unipotent part $x$ of $z$. Moreover, they satisfy $c_{rs}^{(ij)}(x)\ll 1$ and $c_{ij}^{(ij)}(x) = 1$.
\end{lemma}

\begin{proof}
Writing $z=xy$ with $x=(x_{ij})\in\U_n(\RR)$ and $y$ as in \eqref{ydef}, we have $z^{-1}\gamma z = y^{-1}(x^{-1}\gamma x)y$.
Since both $x$ and $x^{-1}$ are upper-triangular, the factor $x^{-1}\gamma x$ in the middle has entries
\[(x^{-1}\gamma x)_{ij} = \sum_{r=1}^n \sum_{s=1}^n x^{-1}_{ir} \gamma_{rs} x_{sj} = \sum_{r=i}^n\sum_{s=1}^j c_{rs}^{(ij)}(x) \gamma_{rs},\]
where $c_{rs}^{(ij)}:=x^{-1}_{ir}x_{sj}$, and $x^{-1}_{ir}$ stands for the $(i,r)$-entry of $x^{-1}$. Now $c_{rs}^{(ij)}(x)\ll 1$ follows
from $x^{-1}_{ir}\ll 1$ and $x_{sj}\ll 1$ (a consequence of \eqref{siegel}), while $c_{ij}^{(ij)}(x)=1$ follows from $x^{-1}_{ii}=1$ and $x_{jj}=1$.

By \eqref{ydef}, right-multiplication by $y$ multiplies the $j$-th column by $y_1\dots y_{n-j}$, and left-multiplication by $y^{-1}$ multiplies the $i$-th row by $(y_1\dots y_{n-i})^{-1}$. The result follows.
\end{proof}

For future reference, we record the following simple consequences of \eqref{Ydef} and \eqref{siegel}:
\begin{equation}\label{Yijproperty}
Y_{ij}\ll Y_{i+1,j}\qquad\text{and}\qquad Y_{ij}\ll Y_{i,j-1}.
\end{equation}
By Lemma~\ref{lemma: entries of z(-1) gamma z}, we have
\begin{equation}\label{eq:gamma_0}
\gamma_{ij}\ll\Gamma_{ij}Y_{ij}^{-1}+\underset{(r, s) \neq (i, j)}{\sum_{r=i}^n\sum_{s=1}^j} |\gamma_{rs}|.
\end{equation}
We claim that this implies
\begin{equation}\label{eq:gamma_1}
\gamma_{ij}\ll Y_{ij}^{-1}.
\end{equation}
Indeed, proceeding by induction on the difference $j-i$, we can assume that $\gamma_{rs}\ll Y_{rs}^{-1}$ holds whenever $s-r<j-i$. Then, combining \eqref{eq:gamma_0} with \eqref{Yijproperty} and the obvious bound $\Gamma_{ij}\ll 1$, we obtain \eqref{eq:gamma_1} readily:
\[\gamma_{ij}\ll\sum_{r=i}^n\sum_{s=1}^j Y_{rs}^{-1}\ll Y_{ij}^{-1}.\]

For $1\leq m\leq n-1$, we denote by $P_m\leq\PGL_n(\RR)$ the maximal parabolic subgroup consisting of the matrices with vanishing lower left $m\times(n-m)$ block. To ease working with this subgroup, we denote by $B_m$ the corresponding set of pairs $\{n-m+1,\dots,n\}\times\{1,\dots,n-m\}$. Further, for any $M\subseteq\{1,\dots,n-1\}$, we write
\[P_M^*:=\bigcap_{m\in M}P_m\cap\bigcap_{m\notin M}P_m^c\qquad\text{and}\qquad
B_M:=\bigcup_{m\in M}B_m,\]
where $P_m^c$ stands for the complement of $P_m$ in $\PGL_n(\RR)$.

\begin{lemma}\label{lemma:observation} Assume that $\gamma\in\PGL_n(\ZZ)$ satisfies $z^{-1}\gamma z\in\Omega$. Let $1\leq m\leq n-1$ be arbitrary.
\begin{itemize}
\item[(a)] If $\gamma\in P_m^c$, then $y_m\ll 1$ holds.
\item[(b)] If $\gamma\in P_m$, then $\Gamma_{ij}\ll\|C(z^{-1}\gamma z)\|$ holds for any $(j,i)\in B_m$.
\end{itemize}
\end{lemma}

\begin{proof} Let $\gamma\in P_m^c$. By definition, $\gamma_{ij}\neq 0$ holds for some $(i,j)\in B_m$. Therefore, $n-i+1\leq m\leq n-j$, and by \eqref{eq:gamma_1}, \eqref{Ydef}, \eqref{siegel}, we infer that
\[1\leq |\gamma_{ij}| \ll Y_{ij}^{-1} = (y_{n-i+1}\dots y_{n-j})^{-1} \ll y_m^{-1}.\]
Comparing the two sides, the bound $y_m\ll 1$ follows.

\medskip
\noindent
Let $\gamma\in P_m$. Writing $g:=z^{-1}\gamma z$, we see that $g\in P_m\cap\Omega$ and $g^{-1}\in P_m\cap\Omega^{-1}$, because $P_m$ is a subgroup containing the upper-triangular matrix $z\in\cH_n$. In addition, as $\Omega$ and $\Omega^{-1}$ are compact,
\[\|C(g)\|\asymp\|g^tg-1_n\|\asymp\|g^t-g^{-1}\|.\]
For any $(j,i)\in B_m$, the $(j,i)$-entry of $g^{-1}$ vanishes, hence $(g^t)_{ji}\ll\|C(g)\|$ follows from the previous display. In other words, we have $g_{ij}\ll\|C(g)\|$ in this case.
\end{proof}

For a subset $M\subseteq\{1,\dots,n-1\}$ and a dyadic parameter $\lambda_{\phi}^{-1/2}\leq K\ll 1$, let us examine the contribution to
\eqref{eq:pre-trace_inequality} of the matrices $\gamma\in\PGL_n(\ZZ)$ satisfying $\gamma\in P^*_M$ and $\|(C(z^{-1}\gamma z)\|\asymp K$. Using \eqref{eq:inverse_spherical_transform_BP_bound}, we obtain
\begin{equation}\label{eq:dyadically_decomposed}
\sum_{\substack{\gamma\in P^*_M\cap\PGL_n(\ZZ) \\K/2<\|(C(z^{-1}\gamma z)\|\leq K}} f_{\mu}(z^{-1}\gamma z) \ll
\lambda_{\phi}^{(n^2-n-1)/4} K^{-1/2}\sum_{\substack{\gamma \in P^*_M\cap \PGL_n(\ZZ) \\\|(C(z^{-1}\gamma z)\|\leq K}} 1.
\end{equation}
To bound the number of $\gamma$'s on the right hand side, we estimate the entries $\gamma_{ij}$ in terms of $M$ and $K$. Note that we have dropped the lower bound constraint for $\|(C(z^{-1}\gamma z)\|$ for convenience.

\emph{Case 1.} Let $i\geq j$. Then \eqref{eq:gamma_1}, \eqref{Ydef}, \eqref{siegel} imply $\gamma_{ij}\ll 1$.

\emph{Case 2.} Let $i<j$ and $(j,i)\notin B_M$. Then \eqref{eq:gamma_1} reads $\gamma_{ij}\ll y_{n-j+1}\dots y_{n-i}$. The interval $n-j+1\leq m\leq n-i$ contains no element from $M$, hence by Part~(a) of Lemma~\ref{lemma:observation}, each corresponding factor $y_m$ can be bounded as $y_m\ll 1$. Therefore, $\gamma_{ij}\ll 1$ as in the previous case.

\emph{Case 3.} Let $(j,i)\in B_M$. Then, $\Gamma_{ij}\ll K$ by Part~(b) of Lemma~\ref{lemma:observation}. We can use this information to strengthen \eqref{eq:gamma_1} to $\gamma_{ij}\ll 1+KY_{ij}^{-1}$. Indeed, proceeding by induction on the difference $j-i$, we can assume that $\gamma_{rs}\ll 1+KY_{rs}^{-1}$ holds whenever $(s,r)\in B_M$ and $s-r<j-i$. Taking into account Cases~1--2 above, $\gamma_{rs}\ll 1+KY_{rs}^{-1}$ holds whenever $s-r<j-i$. Then, combining \eqref{eq:gamma_0} with \eqref{Yijproperty} and the bound $\Gamma_{ij}\ll K$, we obtain the promised improvement of \eqref{eq:gamma_1}:
\[\gamma_{ij} \ll KY_{ij}^{-1} + \underset{(r, s) \neq (i, j)}{\sum_{r=i}^n \sum_{s=1}^j} |\gamma_{rs}| \ll \sum_{r=i}^n \sum_{s=1}^j (1 + KY_{rs}^{-1}) \ll 1 + KY_{ij}^{-1}.\]

In the end, we can bound the number of relevant $\gamma$'s on the right hand side of \eqref{eq:dyadically_decomposed} as
\begin{equation}\label{gammacount}
\sum_{\substack{\gamma \in P^*_M\cap \PGL_n(\ZZ) \\\|(C(z^{-1}\gamma z)\|\leq K}} 1\ll\prod_{(j,i)\in B_M} (1+K Y_{ij}^{-1}).
\end{equation}
The remaining contribution of $\|(C(z^{-1}\gamma z)\|\leq\lambda_\phi^{-1/2}$ can be handled similarly, namely \eqref{eq:inverse_spherical_transform_BP_bound} and \eqref{gammacount} for $K:=\lambda_\phi^{-1/2}$ yield readily
\begin{equation}\label{eq:remaining}
\sum_{\substack{\gamma\in P^*_M\cap\PGL_n(\ZZ) \\\|(C(z^{-1}\gamma z)\|\leq\lambda_\phi^{-1/2}}} f_{\mu}(z^{-1}\gamma z) \ll
\lambda_{\phi}^{(n^2-n)/4}\prod_{(j,i)\in B_M} (1+\lambda_\phi^{-1/2}Y_{ij}^{-1}).
\end{equation}

\subsection{Concluding Theorem~\ref{thm3}}\label{sub4} Summing up over all subsets $M\subseteq\{1,\dots,n-1\}$ and all dyadic parameters $\lambda_{\phi}^{-1/2}\leq K\ll 1$, we obtain by \eqref{siegel}, \eqref{eq:pre-trace_inequality}--\eqref{Ydef}, \eqref{eq:dyadically_decomposed}--\eqref{eq:remaining} the uniform bound
\[|\phi(z)|^2 \ll_n \lambda_{\phi}^{(n^2-n)/4} + \lambda_{\phi}^{(n^2-n-1)/4} \prod_{1\leq i<j\leq n} Y_{ij}^{-1}
=\lambda_{\phi}^{(n^2-n)/4} + \lambda_{\phi}^{(n^2-n-1)/4} \prod_{k=1}^{n-1} y_k^{k(n-k)},\]
which is equivalent to \eqref{thm3bound}. Concerning \eqref{thm4bound}, we need to prove that
\[\phi(z)\ll_{n,\eps}\lambda_\phi^{(n^2-2)(n+1)/16+\eps}\]
holds for any point $z\in\cH_n$ in the Siegel set \eqref{siegel}. If $y_1\dots y_{n-1}>\lambda_\phi^{(n-1)/2+\eps}$ holds for a given $\eps>0$, then \eqref{thm2bound} combined with \eqref{Tdef} and \eqref{cYbound} readily gives $\phi(z)\ll_{n,\eps}1$, with a lot to spare. On the other hand, if $y_1\dots y_{n-1}\leq\lambda_\phi^{(n-1)/2+\eps}$ holds, then \eqref{thm3bound} yields the claimed bound upon noting that $i(n-i)\leq n^2/4$:
\[\phi(z)\ll_{n,\eps}\lambda_\phi^{(n^2-n-1)/8+n^2(n-1)/16+\eps}=\lambda_\phi^{(n^2-2)(n+1)/16+\eps}.\]
The proof of Theorem~\ref{thm3} is complete.

\end{document}